\documentclass[11pt]{article}

\usepackage{amsmath,amsthm}

\usepackage{amssymb,latexsym}

\usepackage{enumerate}

\newcommand{\BB}{{\mathcal B}}
\newcommand{\CC}{{\mathcal C}}
\newcommand{\DD}{{\mathcal D}}

\newcommand{\FF}{{\mathcal F}}

\newcommand{\LL}{{\mathcal L}}

\newcommand{\TT}{{\mathcal{T}}}

\newcommand{\BR}{{\mathbb R}}

\newcommand{\fch}{{\mathbf{1}}}

\newtheorem{theorem}{\bf Theorem}[section]
\newtheorem{proposition}[theorem]{\bf Proposition}
\newtheorem{lemma}[theorem]{\bf Lemma}
\newtheorem{corollary}[theorem]{\bf Corollary}

\theoremstyle{definition}
\newtheorem{definition}{\bf Definition}[section]
\newtheorem{example}[theorem]{\bf Example}
\newtheorem*{fakt}{Fakt}

\theoremstyle{remark}
\newtheorem{remark}{\bf Remark}[section]

\begin{document}

\title {On perpetual American options in a multidimensional Black-Scholes
model}
\author {Andrzej Rozkosz
\mbox{}\\[2mm]
{\small  Faculty of Mathematics and Computer Science,
Nicolaus Copernicus University}\\
{\small Chopina 12/18, 87-100 Toru\'n, Poland} \\
{\small E-mail address: rozkosz@mat.umk.pl}}
\date{}
\maketitle

\begin{abstract}
We consider the problem of pricing perpetual American options written on dividend-paying assets whose price dynamics follow a multidimensional Black and Scholes model. For convex  Lipschitz continuous  reward functions, we give a probabilistic  characterization of the fair price in terms of a reflected BSDE, and an analytical one  in terms of an obstacle problem. We also provide the  early exercise premium formula.
\end{abstract}


\section{Introduction}

\begin{fakt}
 to jest
\end{fakt}

In this paper, we consider the problem of pricing perpetual American options written on dividend-paying assets whose price dynamics follow the classical multidimensional Black and Scholes model. In this model, under the risk-neutral measure $P$, the asset prices $X^{s,x,1},\dots,X^{s,x,d}$  on $[s,\infty)$ evolve according to the stochastic differential equation
\begin{equation}
\label{eq1.1}
X^{s,x,i}_{t}=x_i+\int_{s}^{t}(r-\delta_{i})X^{s,x,i}_{\theta}\,d\theta
+\sum_{j=1}^n\int_{s}^{t}\sigma_{ij}X^{s,x,i}_{\theta}\,dW^j_{\theta},
\quad t\ge s.
\end{equation}
In (\ref{eq1.1}), $W$ is a standard  $d$-dimensional Wiener process, $x_i>0$, $i=1,\dots,d$, are the initial prices at time $s$, $r>0$ is the risk-free interest rate, $\delta_i\ge0$, $i=1,\dots,d$, are dividend rates and $\sigma=\{\sigma_{ij}\}_{i,j=1,\dots,d}$ is the volatility matrix. We assume that $a=\sigma\cdot\sigma^*$, where $\sigma^*$  is the transpose of $\sigma$, is strictly   positive definite.

Let $T>0$ and  $\psi:\BR^d\rightarrow\BR$ be a nonnegative continuous function satisfying the linear growth condition.  Under the measure $P$, the value at time $s$ of the American option with payoff function $\psi$ and expiration time $T>0$ is given by
\begin{equation}
\label{eq1.2}
V_T(s,x)=\sup_{s\le\tau\le T}Ee^{-r(\tau-s)}\psi(X^{s,x}_{\tau}),
\end{equation}
and the value of the perpetual option with payoff function $\psi$ is
\begin{equation}
\label{eq1.3}
V(s,x)=\sup_{\tau\ge s}Ee^{-r(\tau-s)}\psi(X^{s,x}_{\tau})
\end{equation}
(see \cite{Ka,KS,S}). In (\ref{eq1.2}), the supremum  is taken
over the set of all stopping times with values in $[s,T]$, and in
(\ref{eq1.3}), over the set of stopping times in $[s,\infty]$. In
the event that $\tau=\infty$, we interpret
$e^{-r(\tau-s)}\psi(X^{s,x}_{\tau})
=\varlimsup_{t\rightarrow\infty}e^{-r(t-s)}\psi(X_t)$.

Nowadays, properties of $V_T$ are quite well investigated. It is
known (see  \cite{EKPPQ,EPQ,EQ}) that $V_T$ can be represented by
a solution  of a reflected backward stochastic differential
equation (RBSDE).  A detailed study of the structure of this
RBSDE, which in particular leads to the  early exercise premium
formula,  is given in \cite{KR:AMO} (also see Section
\ref{sec3.1}). The value $V_T$ can also be characterized
analytically as a solution of some obstacle problem (or, in
different terminology, variational inequality) (see
\cite{EKPPQ,EPQ,EQ,KR:AMO} and Section \ref{sec3.2}). It is worth
noting here that the analytical characterization relies heavily on
the characterization via solutions of RBSDEs.

In the case of perpetual options less in known, except for put  and
call options in case $d=1$, which were thoroughly investigated as
early as in \cite{McK,M}. For a nice presentation of these results
as well as some newer results and historical comments we refer the
reader to the books \cite{KS,S}. Presumably, the main reason that
less attention has been paid to $V$ than to $V_T$ is that
perpetual options are not traded. On the other hand, in our
opinion, perpetual American options are interesting for
historical reasons and from a purely theoretical point of view.
This motivated us to ask whether in the multidimensional case
one can represent $V$ in terms
of BSDEs or solutions of  obstacle problems. The answer is ``yes''
and the desired  representations of $V$ can
be derived in a quite elegant way from those of $V_T$. The main
idea is as follows. Intuitively, $V$ is the limit of $V_T$ as
$T\rightarrow\infty$
This suggests that properties of $V$ we are
interested in can be derived by studying the behaviour,  as
$T\rightarrow\infty$,  of the solution of the RBSDE with terminal
condition at time $T$, which is used to represent $V_T$. By
modifying some results from the recent paper \cite{KR:PA}, we show
that the  idea sketched above is  realizable. As a result we
show that for convex and Lipschitz continuous $\psi$ the value
function $V$ is represented  by a solution of some RBSDE with
terminal condition 0 at infinity and we get the exercise premium
formula. We also show that $V$ is a unique solution of some
obstacle problem. Finally, we estimate that rate of convergence of
$V_T $ to $V$. It seems that some of our results (the
representation in terms of RBSDEs, rate of convergence) are new
even in the case of the classical call/put option and $d=1$.

\section{Preliminaries}

Let $\Omega=C([0,T];\BR^d)$ and  $X$ be the canonical process on
$\Omega$. For $(s,x)\in[0,T]\times\BR^d$ let $P_{s,x}$  denote the
law of the process $X^{s,x}=(X^{s,x,1},\dots,X^{s,x,d})$ defined
by (\ref{eq1.1}) and  $\{\FF^s_t\}$  denote the completion of
$\sigma(X_{\theta};\theta\in[s,t])$ with respect to the family
$\{P_{s,\mu};\mu$ a finite measure on $\BB(\BR^n)\}$, where
$P_{s,\mu}(\cdot)=\int_{\BR^d}P_{s,x}(\cdot)\,\mu(dx)$. Let
$a=\sigma\cdot\sigma^*$. Using It\^o's formula and L\'evy's
characterization of the Wiener process  one can check (see
\cite[Section 2]{KR:AMO} for details) that
\begin{equation}
\label{eq2.1}
X^i_t=x_i+\int_{s}^{t}(r-\delta_{i})X^{i}_{\theta}\,d\theta
+\sum_{j=1}^d\int_{s}^{t}\sigma_{ij}X^{i}_{\theta}\,dB^j_{s,\theta},
\quad t\ge s,\quad P_{s,x}\mbox{-a.s.},
\end{equation}
where,  under the measure $P_{s,x}$, $\{B_{s,t}, t\ge s\}$ is a standard
$d$-dimensional $\{\FF^s_t\}$-Wiener process on $[s,\infty)$. It
is well known that  the unique solution of (\ref{eq2.1}) is of the
form
\begin{equation}
\label{eq2.2} X^{i}_t=x_i\exp\big((r-\delta_i-a_{ii}/2)(t-s)
+\sum^d_{j=1}\sigma_{ij}B^j_{s,t}\big),\quad t\ge s,\quad
P_{s,x}\mbox{-a.s.}
\end{equation}
Since  $\tilde B^i:=\sum^d_{j=1}\sigma_{ij}B^j_{s,\cdot}$ is a
continuous martingale with the quadratic variation $\langle \tilde
B^i_{s,\cdot}\rangle_t=a_{ii}(t-s)$, $t\ge s$, the process $X^i$
can be written as
\begin{equation}
\label{eq2.3}
X^i_t=x_ie^{(r-\delta_i)(t-s)}N^i_{s,t}, \quad t\ge s,
\end{equation}
where
\[
N^i_{s,t}=\exp(-(t-s)a_{ii}/2+\tilde B^i_{s,t})=\exp(-\langle\tilde B^i_{s,\cdot}\rangle/2+\tilde B^i_{s,\cdot}),\quad t\ge s,
\]
is an $(\FF^s_t)$-martingale under $P_{s,x}$. Let
$D=\{x=(x_1,\dots,x_d):x_i>0,i=1,\dots,d\}$. From (\ref{eq2.2}) it
follows that if $x\in D$, then $P_{s,x}(X_t\in D,t\ge s)=1$.

Below  we recall some known results on the pricing of American options with finite expiration time $T>0$. They will be needed in the next section.

In this paper, we assume that the payoff function  satisfies the following condition:
\begin{enumerate}
\item[(A1)]
$\psi:\BR^d\rightarrow\BR$ is a nonnegative convex function which is Lipschitz continuous, i.e. there is $L>0$ such that $|\psi(x)-\psi(y)|\le L|x-y|$ for all $x,y\in\BR^d$.
\end{enumerate}
In particular,
\begin{equation}
\label{eq2.4}
\psi(x)\le C(1+|x|),\quad x\in \mathbb{R}^d,
\end{equation}
with $C=\max\{L,\psi(0)\}$. Furthermore, since $\psi$ is convex,
for a.e. $x\in\BR^d$ there exist the usual partial derivatives
$\nabla_1\psi(x), \dots,\nabla_d\psi(x)$ of $\psi$ at $x$.
Furthermore, by Alexandrov's theorem (see, e.g., \cite[Theorem
7.10]{AA}), there is a set $N\subset\BR^d$ of Lebesgue measure zero such that
$\psi$ has  second order derivatives at $x$ for every
$x\in\BR^d\setminus N$. We denote them by  $\nabla^2_{ij}\psi(x)$.

Let $\TT_{s,T}$ denote the set  of all $(\FF^s_t)$-stopping times with values in $[s,T]$. The  fair price (or value) $V_T(s,x)$ of the American option with  expiration time $T$ and payoff function $\psi$ is given by
\begin{equation}
\label{eq5.1}
V_T(s,x)=\sup_{\tau\in\TT_{s,T}}E_{s,x}e^{-r(\tau-s)}\psi(X_{\tau}).
\end{equation}

Let $L=\psi(X)$. Note that $E_x|N^i_{s,T}|^2=e^{a_{ii}(T-s)}$,  so
by (\ref{eq2.3}) and Doob's inequality, $E_{s,x}\sup_{s\le t\le
T}|X^i_t|^2<\infty$, $i=1,\dots,d$. By this and (\ref{eq2.4}),
$E_{s,x}\sup_{s\le t\le T}|L_t|^2<\infty$. Therefore, by
\cite[Theorem 5.2]{EKPPQ},   for every $(s,x)\in[0,T]\times\BR^d$
there exists a unique solution $(Y^{T,s,x},K^{T,s,x},Z^{T,s,x})$,
on the space $(\Omega,\FF^s_T,P_{s,x})$, of the RBSDE  with
coefficient $f(y)=-ry$, $y\in\BR$, terminal condition $\psi(X_T)$
and barrier $L$, that is linear RBSDE of the form
\begin{equation}
\label{eq5.2}
\left\{
\begin{array}{l}
Y^{T,s,x}_t=\psi(X_T)-\int^T_trY^{T,s,x}_{\theta}\,d\theta
+\int^T_tdK^{T,s,x}_{\theta}
-\int^T_tZ^{T,s,x}_{\theta}\,dB_{s,\theta},\,\,
t\in[s,T],\medskip\\
Y^{T,s,x}_t\ge \psi(X_t),\quad t\in[s,T],\medskip \\
K^{T,s,x}_0=0\,\, ,K^{T,s,x}\mbox{ is continuous and  increasing, and satisfies }\\
\qquad\qquad\mbox{the minimality condition }
\int^T_s(Y^{T,s,x}_t-\psi(X_t))\,dK^{T,s,x}_t=0.
\end{array}
\right.
\end{equation}
For a precise definition of a solution we refer the reader to
\cite{EKPPQ}. Here let us only note that
$E_{s,x}\int^T_s|Z^{T,s,x}_{\theta}|^2\,d\theta<\infty$, so the
process
\[
M^{T,s,x}_t=\int^t_sZ^{T,s,x}_{\theta}\,dB_{s,\theta},\quad t\in[s,T],
\]
is a martingale under $P_{s,x}$.
Let $L_{BS}$ denote the Black-Scholes operator defined by
\[
L_{BS}=\sum^d_{i=1}(r-\delta_i)x_i\partial_{x_i}
+\frac12\sum^d_{i,j=1}a_{ij}x_ix_j\partial^2_{x_ix_j}\,,
\]
where $\partial_{x_i},\partial^2_{x_ix_j}$ denote the partial derivatives in the distribution sense. In \cite[Theorem 8.5]{EKPPQ} it is also proved that for every $(s,x)\in [0,T]\times\BR^d$,
\begin{equation}
\label{eq2.8}
Y^{T,s,x}_t=u_T(t,X_t),\quad t\in[s,T],\quad
P_{s,x}\mbox{-a.s.},
\end{equation}
where $u_T$ is the unique viscosity solution to the obstacle problem
\begin{equation}
\label{eq5.3} \left\{
\begin{array}{ll}
\min\{u_T-\psi,-\partial_su_T-L_{BS}u_T+r u_T\}=0 &\mbox{in }[0,T]\times\BR^d,
\medskip\\
u_T(T,\cdot)=\psi & \mbox{on } x\in\BR^d.
\end{array}
\right.
\end{equation}
The process $\bar Y^{T,s,x}$ defined as $\bar Y^{T,s,x}_t=e^{-r(t-s)}Y^{T,s,x}_t$, $t\in[s,T]$, is the first component of the solution of RBSDE with coefficient $f=0$, terminal condition $e^{-rT}\psi(X_T)$ and barrier $e^{-rt}\psi(X_t)$, $t\in[s,T]$.
Therefore  from  (\ref{eq2.8}) with $t=s$ and \cite[Proposition 2.3]{EKPPQ} (or \cite[Proposition 3.3]{EPQ}) it follows that  $V_T=u_T$.
Let
\[
\LL_{BS}=\sum^d_{i=1}(r-\delta_{i})x_{i}\nabla_i +\frac12\sum^d_{i,j=1}
a_{ij}x_ix_j\nabla_{ij}\,.
\]
In  \cite[Theorem 2]{KR:AMO} it is proved that under (A1), for every $(s,x)\in[0,T]\times D$,
\begin{equation}
\label{eq2.10}
K^{T,s,x}_t=\int^t_s\Phi(X_{\theta},u_T(\theta,X_{\theta}))\,d\theta,\quad t\in[s,T],\quad P_{s,x}\mbox{-a.s.},
\end{equation}
where
\begin{equation}
\label{eq2.9}
\Phi(x,y)=\Psi^{-}(x)\fch_{(-\infty,\psi(x)]}(y),\qquad \Psi^-=\max\{-\Psi,0\}
\end{equation}
and
\begin{equation}
\label{eq2.11}
\Psi(x)=-r\psi(x)+\LL_{BS}\psi(x)\quad \mbox{if }x\in D\setminus N,\qquad
\Psi(x)=0\quad\mbox{if }x\in N.
\end{equation}
Note that from (\ref{eq2.2}) it follows that if $(s,x)\in[0,\infty)\times D$ and $t\in(s,T]$, then under the measure $P_{s,x}$ the random variable $X_t$ has density with respect to the Lebesgue measure. Therefore $K^{T,s,x}$ is independent of $N$ and  the way we define $\Psi$ on $N$. Note also that
\[
\Phi(x,0)=\Psi^{-}(x),\qquad\Phi(x,u_T(s,x))=\Psi^{-}(x)\fch_{\{u_T(s,x)=\psi(x)\}}, \quad (s,x)\in[0,T]\times D,
\]
since $u_T(s,x)\ge\psi(x)\ge0$,

\section{Perpetual options}
\label{sec3}

To shorten notation, in this section we set $V(x)=V(0,x)$,  $\FF_t=\FF^0_t$,
$P_x=P_{0,x}$, and we denote by $E_x$ the expectation with respect to $P_x$. With this notation (\ref{eq1.3}) takes the form
\begin{equation}
\label{eq6.7}
V(x)=\sup_{\tau\in\TT}E_xe^{-r\tau}\psi(X_{\tau}),
\end{equation}
where $\TT$ is the set of all $(\FF_t)$-stopping times. In
the event that $\tau=\infty$, we interpret
$e^{-r(\tau-s)}\psi(X_{\tau})
=\varlimsup_{t\rightarrow\infty}e^{-r(t-s)}\psi(X_t)$.

\subsection{Stochastic representation of the value function}
\label{sec3.1}

Assume (A1) and let
\[
Y^T_t=u_T(t,X_t),\qquad K^{T}_t=\int^t_0\Phi(X_{s},u_T(s,X_{s}))\,ds,\quad t\in[0,T].
\]
By (\ref{eq2.8}) and (\ref{eq2.10}),  $Y^T$ and $K^T$ are
independent of $x$ versions of $Y^{T,0,x}$ and $K^{T,0,x}$, respectively.
Since $V_T=u_T$, we have
\begin{equation}
\label{eq5.9}
V_T(t,X_t)=Y^{T}_t=u_T(t,X_t),\quad t\in[0,T], \quad
P_{x}\mbox{-a.s.}
\end{equation}
By the first equation in (\ref{eq5.2}) we have
\[
M^{T,0,x}_t=Y^{T,0,x}_t-Y^{T,0,x}_0-\int^t_0rY^{T,0,x}_s\,ds+K^{T,0,x}_t,\quad t\ge0,
\]
so $M^{T,0,x}$ also has a version independent of $x$, which we denote by $M^T$.
Set
\[
\bar Y^{T}_t=e^{-rt}Y^T_t,\qquad \bar K^{T}_t=\int^t_0e^{-rs}\,dK^T_s,\qquad \bar M^T_t=\int^t_0e^{-rs}\,dM^T_s,\quad t\in[0,T].
\]
Since
\[
Y^{T}_t=\psi(X_T)-\int^T_trY^{T}_{s}\,ds
+\int^T_tdK^{T}_s-\int^T_tdM^T_s,\quad t\in[s,T],
\]
integrating by parts we obtain
\begin{equation}
\label{eq6.4}
\bar Y^T_t=e^{-rT}\psi(X_T)+\int^T_td\bar K^T_s
-\int^T_td\bar M^T_s,\quad t\in[0,T].
\end{equation}

We will also need the following condition.
\begin{enumerate}
\item[(A2)]For every $x\in D$,
\begin{equation}
\label{eq3.13}
\mbox{\rm(a)}\,\,\lim_{t\rightarrow\infty}E_xe^{-rt}\psi(X_t)=0,\qquad \mbox{\rm(b)}\,\,E_x\int^{\infty}_0e^{-rt}\Psi^{-}(X_t)\,dt<\infty.
\end{equation}
\end{enumerate}

\begin{remark}
\label{rem3.1}
(i) Condition (\ref{eq3.13}) can be equivalently stated as
\[
\mbox{\rm(a)}\,\,\lim_{t\rightarrow\infty}e^{-rt}P_t\psi(x)=0,\qquad \mbox{\rm(b)}\,\,R_r\Psi^{-}(x)<\infty,
\]
where $(P_t)_{t>0}$ (resp. $(R_{\alpha})_{\alpha>0})$ is the semigroup (resp. resolvent) associated with $X$.
\smallskip\\
(ii) Assume that $r>0$. Clearly (\ref{eq3.13})(a) is satisfied for
all $x\in D$  if $\psi$ is bounded. By (\ref{eq2.3}),
$E_xX^{i}_t=x_ie^{(r-\delta_i)t}$, $t\ge0$. Therefore
(\ref{eq3.13})(a) is satisfied, for all $x\in D$, for general
Lipschitz continuous $\psi$ if $\delta_i>0$, $i=1,\dots, d$.
Similarly, (\ref{eq3.13})(b) is satisfied, for all $x\in D$,  if
$\Psi^{-}$ is bounded or $\delta_i>0$, $i=1,\dots, d$, and there
is $c>0$ such that
\begin{equation}
\label{eq3.42} \Psi^{-}(x)\le c(1+|x|),\quad x\in\BR^d.
\end{equation}
\end{remark}

We are going to show that if (\ref{eq3.13}) is satisfied for some $x\in D$, then $\bar Y^T$ converges as $T\rightarrow\infty$ to a process $\bar Y^x$ being the first component of the solution $(\bar Y^x,\bar K^x,\bar M^x)$ of the reflected BSDE which informally can be
written as
\begin{equation}
\label{eq6.3}
\bar Y^x_t=\int^{\infty}_td\bar K^x_s
-\int^{\infty}_td\bar M^x_s,\quad t\ge0.
\end{equation}
We will also show that $\bar K^x$ has the representation
\begin{equation}
\label{eq6.6}
\bar K^x_t=\int^t_0e^{-rs}\Phi(X_s,e^{rs}\bar Y^x_s)\,ds,\quad t\ge0,
\end{equation}
so in fact $(\bar Y^x,\bar M^x)$ is a solution of the usual BSDE
\begin{equation}
\label{eq3.5}
\bar Y^x_t=\int^{\infty}_te^{-rs}\Phi(X_s,e^{rs}\bar Y^x_s)\,ds
-\int^{\infty}_td\bar M^x_s,\quad t\ge0.
\end{equation}
Equation  (\ref{eq6.3}) is a very special case of nonlinear
reflected BSDEs treated in \cite{HLW}. For existence and
uniqueness results for general infinite horizon BSDEs with
$L^2$-data we refer the reader to \cite{DP} (equations in $\BR^d$)
and \cite{FT} (equations in Hilbert spaces). Roughly speaking, in
\cite{HLW} it is proved that if the coefficient $f$ of the
equation satisfies a generalized Lipschitz condition, its terminal
condition  is square-integrable and its  barrier $\bar L$ is
continuous and satisfies the condition $E_x\sup_{t\ge0}|\bar
L_t|^2<\infty$, then the equation has a unique square-integrable
solution. In (\ref{eq6.3}), the coefficient  $f$ is equal to zero,
terminal condition is equal to zero and the barrier has the form
$\bar L_t=e^{-rt}\psi(X_t)$, $t\ge0$. In general, under (A1) and
(A2), this barrier does not satisfy the aforementioned assumption
from \cite{HLW}, so the results of \cite{HLW} are not directly
applicable to our situation. Assumption (A2)(a) says that $\bar
L_t=e^{-rt}\psi(X_t)$, $t\ge0$, has the property that
$\lim_{t\rightarrow\infty}E_x\bar L_t=0$, $x\in D$. We  shall see
that this condition on $\bar L$ together with (A2)(b) guarantee
the existence of a unique solution of (\ref{eq6.3}) such that its
first component $\bar Y^x$ is of Doob's class (D),  i.e. it has
(in general) weaker integrability properties than the solution
considered in \cite{HLW}.

Before giving the definition of  solutions of (\ref{eq6.3})  and
(\ref{eq3.5}) let us recall that a continuous  $(\FF_t)$-adapted
process $Y$ is said to be of class  (D) under the measure $P_x$
if the collection $\{Y_{\tau}:\tau\in\TT,\tau\mbox{ finite-valued}\}$ is uniformly
integrable under $P_x$. Let $\LL^1(P_x)$ denote the space of
continuous  processes with finite norm
$\|Y\|_{x,1}=\sup\{E_x|Y_{\tau}|:\tau\in\TT, \tau\mbox{ finite-valued}\}$. It is known  that
$\LL^1(P_x)$ is complete (see \cite[Theorem VI.22]{DM}). Moreover,
if $Y^n$ are of class (D) and $Y^n\rightarrow Y$ in $\LL^1(P_x)$,
then $Y$ is of class (D) (see \cite[Section 3]{KR:PA}).

\begin{definition}
(i) We say that a triple $(\bar Y^x,\bar K^x,\bar M^x)$ of adapted continuous processes is a
solution of the reflected BSDE (\ref{eq6.3}) with lower barrier $\bar L_t=e^{-rt}\psi(X_t)$
if $\bar Y^x$ is of class (D), $\bar M^x$ is a local martingale
with $\bar M^x_0=0$, $\bar K^x$ is an increasing process with $\bar K_0=0$, and for every $T>0$,
\begin{equation}
\label{eq6.1}
\left\{
\begin{array}{l}
\bar Y^x_t=\bar Y^x_T+\int^T_td\bar K^x_s
-\int^{T}_td\bar M^x_s,\quad t\ge0,\medskip\\
\bar Y^{x}_t\ge\bar L_t,\quad t\in[0,T],\quad
\int^{T}_0(\bar Y^{x}_t-\bar L_t)\,d\bar K^{x}_t=0, \medskip\\
\bar Y^x_T\rightarrow0\,\,\,P_x \mbox{-a.s. as }T\rightarrow\infty.
\end{array}
\right.
\end{equation}
(ii) We say that a pair  $(\bar Y^x,\bar M^x)$ of adapted continuous processes is a
solution of  the BSDE (\ref{eq3.5}) if $\bar Y^x$ is of class (D), $\bar M^x$ is a local martingale
with $\bar M^x_0=0$, for every $T>0$, $\int^T_0e^{-rt}\Phi(X_t,e^{rt}\bar Y^x_t)\,dt<\infty$ $P_x$-a.s., and moreover,
\begin{equation}
\label{eq3.6}
\left\{
\begin{array}{l}
\bar Y^x_t=\bar Y^x_T+\int^T_te^{-rs}\Phi(X_s,e^{rs}\bar Y^x_s)\,ds
-\int^{T}_td\bar M^x_s,\quad t\ge0, \medskip\\
\bar Y^x_T\rightarrow0\,\,\,P_x \mbox{-a.s. as }T\rightarrow\infty.
\end{array}
\right.
\end{equation}
\end{definition}

\begin{remark}
\label{rem3.2}
Assume that for some $x\in D$ there exists a solution  $(\bar Y^x,\bar M^x,\bar K^x)$  of (\ref{eq6.1}). Then\\
(i) $e^{-r\tau}\psi(X_{\tau})=0$ $P_x$-a.s. on $\{\tau=\infty\}$ because  by our convention, on  the set $\{\tau=\infty\}$ we have
$e^{-r(\tau-s)}\psi(X_{\tau})
=\varlimsup_{t\rightarrow\infty}e^{-r(t-s)}\psi(X_t)\le\lim_{t\rightarrow\infty}\bar Y^x_t=0$ $P_x$-a.s.
\\
(ii) For every $\tau\in\TT$,
\[
E_x\bar Y^x_0\ge E_x\bar Y^x_{\tau}\ge E_xe^{-r\tau}\psi(X_{\tau}).
\]
To see this, consider a localizing sequence $\{\tau_n\}$ for $\bar M^x$.
Since
\[
\bar Y^x_t=\bar Y^x_0-\int^t_0d\bar K^x_s+\int^t_0d\bar M^x_s,\quad t\ge0,
\]
we have $E_x\bar Y^x_0\ge \liminf_{n\rightarrow\infty}E_x\bar Y^x_{\tau\wedge\tau_n}$. Applying Fatou's lemma yields the desired inequalities.
\end{remark}

We start with uniqueness results for BSDE (\ref{eq3.5}) and RBSDE (\ref{eq6.1}).

\begin{proposition}
\label{prop2.1}
Assume that $\psi$  satisfies \mbox{\rm(A1)} and  \mbox{\rm(\ref{eq3.13})} for some $x\in D$. Then there is at most one solution of \mbox{\rm(\ref{eq6.1})}. Similarly,  there is at most one solution of \mbox{\rm(\ref{eq3.5})}.
\end{proposition}
\begin{proof}
Suppose that $(\bar Y^i,\bar K^i,\bar M^i)$, $i=1,2$, are solutions of (\ref{eq6.1}). Write $\bar Y=\bar Y^1-\bar Y^2$, $\bar K=\bar K^1-\bar K^2$, $\bar M=\bar M^1-\bar M^2$. Then, by Remark \ref{rem3.2},
\[
\bar Y_t=\bar Y_0-\int^t_0d\bar K_s+\int^t_0d\bar M_s,\quad t\ge0.
\]
By the Meyer-Tanaka formula (see, e.g., \cite[Theorem IV.68]{P}),
\begin{equation}
\label{eq5.5}
\bar Y^{+}_t\le\bar Y^{+}_T+\int^T_t\fch_{\{\bar Y^1_s>\bar Y^2_s\}}\,d\bar K_s
-\int^T_t\fch_{\{\bar Y^1_s>\bar Y^2_s\}}\,d\bar M_s.
\end{equation}
Since $\bar L_t\le\bar Y^1_t\wedge\bar Y^2_t\le\bar Y^1_t$,
we have  
\[
\int^T_t\fch_{\{\bar Y^1_s>\bar Y^2_s\}}\,d\bar K^1_s
=\int^T_t\fch_{\{\bar Y^1_s>\bar Y^2_s\}}(\bar Y^1_s-\bar Y^2_s)^{-1}(\bar Y^1_s- \bar Y^1_s\wedge\bar Y^2_s)\,d\bar K^1_s\le0.
\]
By the above inequality and (\ref{eq5.5}), $E_x\bar Y^{+}_t\le E_x\bar Y^{+}_T$.  Since $E_x\bar Y^+_T\rightarrow0$ as $T\rightarrow\infty$, 
we see that $\bar Y^+_t=0$, $t\ge0$, $P_x$-a.s. In the same way we show that $(-\bar Y_t)^+=0$, $t\ge0$, $P_x$-a.s. Thus $\bar Y^1=\bar Y^2$. That $\bar M^1=\bar M^2$ and $\bar K^1=\bar K^2$ now follows from uniqueness of the Doob-Meyer decomposition of $\bar Y^1$.

The proof of the second assertion is similar. Suppose that   $(\bar Y^1,\bar M^1)$, $(\bar Y^2,\bar M^2)$ are  solutions of (\ref{eq3.5}). Let $\bar Y=\bar Y^1-\bar Y^2$, $\bar M=\bar M^1-\bar M^2$.
Applying the Meyer-Tanaka formula yields
\[
\bar Y^{+}_t\le\bar Y^{+}_T+\int^T_t\fch_{\{\bar Y^1_s>\bar Y^2_s\}}e^{-rs}
(\Psi(X_s,e^{rs}\bar Y^1_s)-\Psi(X_s,e^{rs}\bar Y^2_s))\,ds
-\int^T_t\fch_{\{\bar Y^1_s>\bar Y^2_s\}}\,d\bar M_s.
\]
But
\begin{align}
\label{eq3.21}
&(\Phi(x,y_1)-\Phi(x,y_2))(y_1-y_2) \nonumber\\
&\qquad
=\Psi^{-}(x)(\fch_{(-\infty,\psi(x)]}(y_1)-\fch_{(-\infty,\psi(x)]}(y_2))(y_1-y_2)\le0,
\end{align}
so $\bar Y^{+}_t\le\bar Y^{+}_T-\int^T_t\fch_{\{\bar Y^1_s>\bar Y^2_s\}}\,d\bar M_s$.
To prove that $\bar Y^1=\bar Y^2$ and $\bar M^1=\bar M^2$ it suffices now to repeat the argument from the proof of the first assertion.
\end{proof}

We show next the existence of a solution of (\ref{eq6.1}). To this end, we need some notation. By (\ref{eq6.4})  with $T=n$,
\begin{equation}
\label{eq3.1}
\bar Y^n_t=e^{- rn}\psi(X_n)+\int^n_te^{-rs}\Phi(X_s,e^{rs}\bar Y^n_s)\,ds-\int^n_td\bar M^n_s,\quad t\in[0,n].
\end{equation}
We put
\[
\tilde Y^n_t=\bar Y^n_t,\quad \tilde M^n_t=\bar M^n_t,\quad t<n,\qquad\tilde Y^n_t=0,\quad \tilde M^n_t=\bar M^n_n,\quad t\ge n.
\]

The proof of the following theorem  is a modification of the proof of \cite[Propositions 4.1, 4.2]{KR:PA}.

\begin{theorem}
\label{th3.6}
Assume that $\psi$ satisfies \mbox{\rm(A1)} and \mbox{\rm(\ref{eq3.13})} for some $x\in D$.  Then  there exists a  unique solution $(\bar Y^x,\bar M^x)$ of \mbox{\rm(\ref{eq3.5})} on $(\Omega,\FF,P_x)$. Moreover,
\begin{equation}
\label{eq3.15}
E_x\int^{\infty}_0e^{-rt}\Phi(X_t,e^{rt}\bar Y^x_t)\,dt
\le 2 E_x\int^{\infty}_0e^{-rt}\Psi^{-}(X_t)\,dt,
\end{equation}
\begin{equation}
\label{eq3.7}
\lim_{n\rightarrow\infty}\|\bar Y^n-\bar Y^x\|_{x,1}=0
\end{equation}
and for every $q\in(0,1)$,
\begin{equation}
\label{eq3.8}
\lim_{n\rightarrow\infty}E_x\sup_{t\ge0}|\bar Y^n_t-\bar Y^x_t|^q=0.
\end{equation}
\end{theorem}
\begin{proof}
Uniqueness follows from Proposition \ref{prop2.1}. The  proof of the existence and (\ref{eq3.15})--(\ref{eq3.8}) is divided
into two steps. \\
Step 1. We shall prove some a priori estimates for the process $\bar Y^n$ and the difference $\delta\tilde Y:=\tilde Y^m-\tilde Y^n$. Specifically, we shall prove that
\begin{equation}
\label{eq5.6}
\|\delta\tilde Y\|_{x,1} \le E_{x}\Big(e^{-rm}\psi(X_m)
+e^{-rn}\psi(X_n)+\int^{m}_{n}e^{-rt}\Psi^{-}(X_{t})\,dt\Big),
\end{equation}
\begin{equation}
\label{eq5.7}
E_{x}\sup_{t\ge0}|\delta\tilde Y_t|^q
\le\frac{1}{1-q}E_x\Big(e^{-rm}\psi(X_m)+e^{-rn}\psi(X_n))
+\int^{m}_{n}e^{-rt}\Psi^{-}(X_t)\,dt\Big)^q
\end{equation}
for every $q\in(0,1)$, and  for every $t\ge0$,
\begin{equation}
\label{eq5.10}
E_x\int_0^{t}e^{-rs}\Phi(X_{s},e^{rs}\bar Y^n_{s})\,ds
\le E_x\Big(\bar Y^n_t +2\int_0^{t}e^{-rs}\Psi^{-}(X_{s})\,ds\Big).
\end{equation}
By (\ref{eq3.1}),
\begin{equation}
\label{eq3.9}
\bar Y^n_t
=\bar Y^n_0-\int^t_0\fch_{[0,n]}(s)e^{-rs}\Phi(X_s,e^{rs}\bar Y^n_s)\,ds +\int^t_0\fch_{[0,n]}(s)\,d\bar M^n_s,\quad t\in[0,n].
\end{equation}
Moreover,
\[
\tilde Y^n_t=\tilde Y^n_0-\int^t_0\fch_{[0,n]}(s)e^{-rs}\Phi(X_s,e^{rs}\tilde Y^n_s)\,ds
+\int^t_0dV^n_s+\int^t_0\fch_{[0,n]}(s)\,d\tilde M^n_s,\quad t\ge0,
\]
where
\[
V^n_t=0,\quad t<n,\qquad V^n_t=-\bar Y^n_n,\quad t\ge n.
\]
Hence
\[
\delta\tilde Y_t=\delta\tilde Y_0+R_t+\int^t_0(\fch_{[0,m]}(s)\,d\tilde M^m_s-\fch_{[0,n]}(s)\,d\tilde M^n_s),\quad t\ge0,
\]
with
\begin{align*}
R_t&=-\int^t_0\fch_{[0,n]}(s) e^{-rs}(\Phi(X_{s},e^{rs}\tilde
Y^m_{s})-\Phi(X_{s},e^{rs}\tilde Y^n_{s}))\,ds\\
&\quad-\int^t_0\fch_{(n,m]}(s)e^{-rs}
\Phi(X_{s},e^{rs}\tilde Y^m_{s})\,ds+\int^t_0d(V^m_{s}-V^n_{s}).
\end{align*}
By the Meyer-Tanaka formula, for $t<m$ we have
\[
|\delta\tilde Y_m|-|\delta\tilde Y_t|\ge\int^m_t\mbox{sign}
(\delta\tilde Y_{s-})\,d(\delta\tilde Y)_{s},
\]
where $\mbox{sign}(x)=1$ if $x>0$ and $\mbox{sign}(x)=-1$ if
$x\le0$. Therefore, for $t<m$,
\[
|\delta\tilde Y_t|=E_x(|\delta\tilde Y_t|\,|\FF_t)\le
E_x\Big(|\delta\tilde Y_m| -\int^m_t\mbox{sign}(\delta\tilde
Y_{s-})\,dR_{s}\,\big|\FF_t\Big).
\]
From this it follows that for $t\in[0,m]$,
\begin{align*}
|\delta\tilde Y_t|&\le E_{x}\Big(|\delta\tilde Y_m|
+\int_t^m\fch_{[0,n]}(s)e^{-rs}\mbox{sign}(\delta\tilde
Y_{s})(\Phi(X_{s},e^{rs}\tilde Y^m_{s})-\Phi(X_{s},e^{rs}\tilde Y^n_{s}))\,ds\\
&\qquad\qquad+\int^m_t\fch_{(n,m]}(s)e^{-rs}
\mbox{sign}(\delta\tilde Y_{s})\Phi(X_{s},e^{rs}\tilde Y^m_{s})\,ds
+|V^m_m|+|V^n_n| \,\big|\FF_t\Big).
\end{align*}
By (\ref{eq3.21}),
\[
\int_t^m\fch_{[0,n]}(s)e^{-rs}\mbox{sign}(\delta
\tilde Y_{s})(\Phi(X_{s},e^{rs}\tilde Y^m_{s})-\Phi(X_{s},e^{rs}\tilde
Y^n_{s}))\,ds\le0.
\]
Since $\tilde Y^n_t=0$ for $t\ge n$, it follows from (\ref{eq3.21}) that
\begin{align*}
&\int^m_t\fch_{(n,m]}(s)e^{-rs} \mbox{sign}(\delta
\tilde Y_{s})\Phi(X_{s},e^{rs}\tilde Y^m_{s})\,ds \\
&\qquad \le \int^m_t\fch_{(n,m]}(s)e^{-rs}
\mbox{sign}(\delta\tilde Y_{s})\Psi^{-}(X_{s})\,ds
\le \int^{m}_{n}e^{-rs} \Psi^{-}(X_{s})\,ds.
\end{align*}
Furthermore,  $\delta\tilde Y_m=0$ and
\[
|V^m_m|+|V^n_n|= |\bar Y^m_m|+|\bar Y^n_n|=e^{-rm}\psi(X_m)
+e^{-rn}\psi(X_n).
\]
Therefore, for $t\in[0,m]$ we have
\begin{align*}
|\delta\tilde Y_t|\le E_{x}\Big(e^{-rm}\psi(X_m)
+e^{-rn}\psi(X_n)+\int^{m}_{n}e^{-rs}\Psi^{-}(X_{s})\,ds
\big|\,\FF_t\Big)=:N_t,
\end{align*}
from which (\ref{eq5.6}) follows. By the above inequality and \cite[Lemma 6.1]{BDHPS},
\[
E_{x}\sup_{0\le t\le m}|\delta\tilde Y_t|^q
\le(1-q)^{-1}(E_{x}N_m)^q,
\]
which shows (\ref{eq5.7}). To prove (\ref{eq5.10}), we first observe that by the
Meyer-Tanaka formula,
\[
E_x|\bar Y^n_t|-E_x|\bar Y^n_0|\ge E_x\int^t_0\mbox{sign}(\bar
Y^n_{s-})\,d\bar Y^n_s.
\]
By the above inequality and (\ref{eq3.9}), for $t<n$ we have
\begin{equation}
\label{eq3.10}
E_x|\bar Y^n_t|-E_x|\bar Y^n_0|
\ge-E_x\int^t_0\fch_{[0,n]}(s) \mbox{sign}(\bar Y^n_{s})
e^{-rs}\Phi(X_{s},e^{rs}\bar Y^n_{s})\,ds.
\end{equation}
On the other hand, for every $t\ge0$,
\begin{align*}
\int^t_0e^{-rs}\Phi(X_{s},e^{rs}\bar Y^n_{s})\,ds
&\le\int^t_0e^{-rs}|\Phi(X_{s},e^{rs}\bar Y^n_{s})-\Phi(X_{s},0)|\,ds
+\int^t_0e^{-rs}\Phi(X_{s},0)\,ds \\
&=-\int^t_0{\mbox{sign}}(\bar Y^n_{s})e^{-rs}
(\Phi(X_{s},e^{rs}\bar Y^n_{s})-\Phi(X_{s},0))\,ds\\
&\quad+\int^t_0e^{-rs}\Phi(X_{s},0)\,ds \\
&\le -\int^t_0{\mbox{sign}}(\bar Y^n_{s})e^{-rs} \Phi(X_{s},e^{rs}
\bar Y^n_{s})\,ds +2\int^t_0e^{-rs}\Psi^{-}(X_{s})\,ds,
\end{align*}
which when combined with (\ref{eq3.10}) proves (\ref{eq5.10}).
\\
Step 2. We will prove the existence of a solution of (\ref{eq3.5}) and (\ref{eq3.7}), (\ref{eq3.8}).
From  (\ref{eq3.13}) and (\ref{eq5.6}) it follows
that $\|\bar Y^n-\bar Y^m\|_{x,1}\rightarrow0$ as
$n,m\rightarrow\infty$. Therefore there exists a process
$Y^x\in\LL^1(P_x)$ of class (D) such that (\ref{eq3.7}) is satisfied.
By (\ref{eq3.13}) and (\ref{eq5.6}),
$\lim_{n,m\rightarrow\infty}E_x\sup_{t\ge0}|\bar Y^n_t-\bar Y^m_t|^q
\rightarrow0$.  Since the space $\DD^q(P_x)$ is complete, the last
convergence and (\ref{eq3.7}) imply that $\bar Y^x\in\DD^q(P_x)$ and
(\ref{eq3.8}) is satisfied. By (\ref{eq5.1}) and (\ref{eq5.9}), $\bar Y^n_t\le\bar Y^{n+1}_t$, $t\ge0$, $P_x$-a.s. By this and (\ref{eq3.8}),
\[
\lim_{n\rightarrow\infty}\fch_{\{e^{rt}\bar Y^n_t\le\psi(X_s)\}}
=\fch_{\{e^{rt}\bar Y_t\le\psi(X_s)\}},\quad t\ge0, \quad P_x\mbox{-a.s.}
\]
Hence
\begin{equation}
\label{eq3.16} \lim_{n\rightarrow\infty}\Phi(X_t,e^{rt}\bar
Y^n_t)=\Phi(X_t,e^{rt}\bar Y_t),\quad t\ge0, \quad P_x\mbox{-a.s.},
\end{equation}
so applying Fatou's lemma we conclude from (\ref{eq5.10})
that for every $T>0$,
\begin{equation}
\label{eq3.14}
E_x\int_0^Te^{-rt}\Phi(X_{t},e^{rt}\bar Y^x_{t})\,dt
\le E_x\Big(\bar Y^x_{T}
+2\int_0^{T}e^{-rt}\Psi^{-}(X_{t})\,dt\Big).
\end{equation}
From (\ref{eq3.8}) it follows that  $\bar Y^x_{T}\rightarrow0$
in probability $P_x$ as $T\rightarrow\infty$. As a consequence,
since $\bar Y^x$ is of class (D), $E_x\bar Y^x_{T}\rightarrow0$.
Letting $T\rightarrow\infty$ in (\ref{eq3.14}),  we therefore
get (\ref{eq3.15}). By (\ref{eq3.1}),
\[
\bar Y^n_t=\bar Y^n_T+\int^T_te^{-rs}\Phi(X_s,e^{rs}\bar Y^n_s)\,ds
-\int^T_td\bar M^n_s,\quad t<T\le n.
\]
Since $\bar M^n$ is a martingale, it follows that
\begin{equation}
\label{eq3.17} \bar Y^n_t=E_x\Big(\bar Y^n_T
+\int^T_te^{-rs}\Phi(X_s,e^{rs}\bar Y^n_s)\,ds\big|\FF_t\Big), \quad
t<T\le n.
\end{equation}
By Doob's inequality and (\ref{eq3.7}),
\begin{equation}
\label{eq3.18} \lim_{n\rightarrow\infty}P_x(\sup_{0\le t\le
T} |E_x(\bar Y^n_T-\bar Y_T|\FF_t)|>\varepsilon)\le\varepsilon^{-1}
\lim_{n\rightarrow\infty} E_x|\bar Y^n_T-\bar Y^x_T|=0.
\end{equation}
By (\ref{eq3.13}), (\ref{eq3.16}) and the dominated convergence
theorem,
\begin{equation}
\label{eq3.19}
\lim_{n\rightarrow\infty}E_x\int^T_0e^{-rs}|\Phi(X_s,e^{rs}\bar
Y^n_s)-\Phi(X_s,0)|\,ds=0.
\end{equation}
From  (\ref{eq3.17})--(\ref{eq3.19}) we deduce that
\[
\bar Y^x_t=E_x\Big(\bar Y^x_T+ \int^T_te^{-rs}\Phi(X_s,e^{rs}\bar
Y_s)\,ds\big|\FF_t\Big).
\]
Letting $T\rightarrow\infty$ and using (\ref{eq3.15}) and the fact
that $\lim_{T\rightarrow\infty}E_x\bar Y_T=0$ yields
\[
\bar Y^x_t=E_x\Big(\int^{\infty}_te^{-rs} \Phi(X_s,e^{rs}\bar
Y^x_s)\,ds\big|\FF_t\Big).
\]
Let $\bar M^x$ be a c\`adl\`ag version of the martingale
\begin{equation}
\label{eq6.8}
t\mapsto E_x\Big(\int^{\infty}_0e^{-rs}\Phi(X_s,e^{rs}\bar Y^x_s)\,ds\big|\FF_t\Big)-\bar Y_0.
\end{equation}
One can check that $(\bar Y^x,\bar M^x)$ is a solution of  (\ref{eq3.5}).
\end{proof}

\begin{remark}
Since $\bar M^x$ is a version of (\ref{eq6.8}), it follows from (\ref{eq3.15}) and (A2)(b) that it is a closed martingale. Hence (see, e.g., \cite[Theorem I.12]{P}), $\bar M^x_{\infty}=\lim_{t\rightarrow\infty}\bar M^x_t$ exists $P_x$-a.s. and $\bar M^x$ is a martinagale on $[0,\infty]$. Therefore (\ref{eq6.3}) is satisfied $P_x$-a.s. and $E_x\bar M^x_{\infty}=E_x\bar M^x_0=0$. As a result,
\begin{equation}
\label{eq6.9}
E_x\bar Y^x_0=E_x\int^{\infty}_0e^{-rt}\Phi(X_t,e^{rt}\bar Y^x_t)\,dt.
\end{equation}
\end{remark}

\begin{corollary}
\label{cor3.6}
Let the assumption of Theorem \ref{th3.6} hold.
\begin{enumerate}[\rm(i)]
\item If $(\bar Y^x,\bar M^x)$ is a solution
of \mbox{\rm(\ref{eq3.5})}, then $(\bar Y^x,\bar K^x,\bar M^x)$
with $\bar K^x$ defined by \mbox{\rm(\ref{eq6.6})} is a solution of \mbox{\rm(\ref{eq6.3})}.

\item Conversely, if $(\bar Y^x,\bar K^x,\bar M^x)$ is a solution of \mbox{\rm(\ref{eq6.3})}, then $\bar K^x$ admits the representation \mbox{\rm(\ref{eq6.6})}.
\end{enumerate}
\end{corollary}
\begin{proof}
To prove (i), we only have to show that $\bar Y^x,\bar K^x$ have
the properties formulated in the second line of (\ref{eq6.1}). By
(\ref{eq3.8}), $\bar Y^x_t\ge\bar L_t$, $t\ge0$, since by
construction we have $\bar Y^n_t\ge\bar L_t$, $t\in[0,n]$, for every
$n\ge1$. Clearly $\bar K^x_0=0$ and $\bar K^x$  is continuous and
increasing. Since we know that $\bar Y^x_t\ge\bar L_t$, $t\ge0$,
from (\ref{eq2.9}) and (\ref{eq6.6}) it follows that
\[
\int^T_0(\bar Y^x_t-\bar L^x_t)\,d\bar K^x_t =\int^T_0(\bar
Y^x_t-e^{-rt}\psi(X_t))e^{-rt}\Psi^{-}(X_t)
\fch_{(-\infty,\psi(X_t)]}(e^{rt}\bar Y^x_t)\,dt=0,
\]
so $\bar K^x$ satisfies the minimality condition.
Part (ii)  follows from  (i) and the first part of Proposition \ref{prop2.1}.
\end{proof}

\begin{corollary}
\label{cor3.7}
Assume that  \mbox{\rm(A1), (A2)} are satisfied. Then
\begin{enumerate}[\rm(i)]
\item $V(x)=E_x\bar Y^x_0$, $x\in D$.  Moreover, $e^{rt}\bar Y^x_t=V(X_t)$, $ t\ge0$, $P_x$-a.s. for every $x\in D$.
\item $\lim_{T\rightarrow\infty}V_T(t,x)=V(x)$ for all $t\ge0$ and $x\in D$. Moreover, for every $x\in D$,
    \begin{equation}
    \label{eq3.40}
    V(x)-V_T(0,x)\le e\Big(e^{-rT}\psi(X_T)+\int^{\infty}_Te^{-rt}\Psi^{-}(X_t)\,dt\Big),\quad T>0.
    \end{equation}
\end{enumerate}
\end{corollary}
\begin{proof}
By (\ref{eq5.1}) and (\ref{eq6.7}), $V_n(0,x)\le V(x)$, $n\ge1$, and by (\ref{eq5.9}) and Theorem \ref{th3.6}, $V_n(0,x)=E_x\bar Y^n_0\nearrow E_x\bar Y^x_0$. Hence $E_x\bar Y^x_0\le V(x)$. On the other hand, by Remark \ref{rem3.2}, $E_x\bar Y^x_0\ge V(x)$, which proves
the first part of (i).
From (\ref{eq2.2}) and (\ref{eq5.1}) it follows that  $V_T(t,x)=V_{T-t}(0,x)$, $t\in[0,T]$, $x\in D$. By (\ref{eq5.9}) and  (\ref{eq3.7}),
$\lim_{T\rightarrow\infty}V_{T-t}(0,x)=\lim_{T\rightarrow\infty}E_x\bar Y^{T-t}_0=E_x\bar Y^x_0$, which equals $V(x)$. This proves the first part of (ii). By (\ref{eq3.8}) and (\ref{eq5.7}), for every $q\in(0,1)$,
\[
|V_T(0,x)-V(x)|\le (1-q)^{-1/q} E_x\Big(e^{-rT}\psi(X_T)+\int^{\infty}_Te^{-rt}\Psi^{-}(X_t)\,dt\Big),\quad T>0.
\]
Letting $q\downarrow0$ yields (\ref{eq3.40}).
Finally, by (ii), for every $x\in D$, $e^{rt}\bar Y^T_t=Y^T_t=V_T(t,X_t)\rightarrow V(X_t)$ $P_x$-a.s. as $T\rightarrow\infty$. On the other hand, by (\ref{eq3.7}) again, $e^{rt}\bar Y^T_t\rightarrow e^{rt}\bar Y^x_t$ $P_x$-a.s. as $T\rightarrow\infty$. Hence $e^{rt}\bar Y^x_t=V(X_t)$ $P_x$-a.s. for every $t\ge0$, which proves the second part of (i)  because  the processes  $t\mapsto e^{rt}\bar Y^x_t$ and $V(X)$ are continuous.
\end{proof}

\begin{remark}
\label{rem3.8} (i) The solution $(\bar Y^x,\bar K^x,\bar M^x)$ of
(\ref{eq3.5}) has a version $(\bar Y,\bar K,\bar M)$ independent
of $x$. Indeed, by Corollary \ref{cor3.7}(i), the process $\bar
Y_t=e^{-rt}V(X_t)$, $t\ge0$, is a version of $\bar Y^x$. By this
and Corollary \ref{cor3.6}(ii),  $\bar
K_t=\int^t_0e^{-rs}\Phi(X_s,V(X_s))\,ds$, $t\ge0$, is a version
of $\bar K^x$. Consequently, by the first equation in
(\ref{eq6.1}), the process $\bar M_t=\bar Y_t-\bar Y_0+\bar K_t$,
$t\ge0$, is a version of $\bar M^x$.
\smallskip\\
(ii) The argument from the proof of
\cite[Proposition 5.6]{KR:MF} shows that  if 
$\psi(x)>0$ for some $x\in D$, then $\{x\in D:V(x)=\psi(x)\}\subset\{x\in D:\psi(x)>0\}$. Therefore $\bar K$ can be written in the form
\[
K_t=\int^t_0e^{-rs}\Psi^{-}(X_s)\fch_{\{V(X_s)=\psi(X_s),\,\psi(X_s)>0\}}\,ds, \quad t\ge0.
\]
\end{remark}

The value of ``perpetual European option" with payoff function $\psi$ is defined as
$V^E(x)=\lim_{T\rightarrow\infty}E_xe^{-rT}\psi(X_T)$. Under the assumption (A2) it is equal to zero. Therefore
the next result can be called the early exercise premium formula for perpetual American options.
This formula extends the corresponding  formula for call option in the classical one-dimensional model (see \cite[Chapter 2, Eq. (6.31)]{KS}).
\begin{corollary}
Assume that \mbox{\rm(A1), (A2)} are satisfied.  Then for every $x\in D$,
\end{corollary}
\begin{equation}
\label{eq3.20}
V(x)=E_x\int^{\infty}_0e^{-rt}\Psi^{-}(X_t)\fch_{\{V(X_t)=\psi(X_t),\,\psi(X_t)>0\}}\,dt.
\end{equation}
\begin{proof}
Follows immediately from (\ref{eq6.9}) and Corollary \ref{cor3.7}(i) and Remark \ref{rem3.8}(ii).
\end{proof}

\begin{lemma}
\label{lem3.8}
Assume \mbox{\rm(A1)}. Then
\begin{enumerate}[\rm(i)]
\item $D\ni x\mapsto V_T(x)$, $D\ni x\mapsto V(x)$ are Lipschitz continuous with constant $L$.
\item For all $x\in D$, $T>0$ and $t\in[0,T]$, $V_T(t,x)\le C(1+|x|)$ with $C=\max\{L,\psi(0)\}$.
\end{enumerate}
\end{lemma}
\begin{proof}
(i) For $y\in D$ set $\tilde X=(\tilde X^1,\dots,\tilde X^d)$, where $\tilde X^i$, $i=1,\dots,d$, is defined  by (\ref{eq2.2}) with $x_i$ replaced by $y_i$. Let $x,y\in D$. By (\ref{eq5.1}),
\[
|V_T(0,x)-V_T(0,y)|\le \sup_{\tau\in\TT_T}E_xe^{-r\tau}|\psi(X_{\tau})-\psi(\tilde X_{\tau})|
\le LE_xe^{-r\tau}|X_{\tau}-\tilde X_{\tau}|.
\]
Define $N^i$ as in (\ref{eq2.3}). Since
$|X^i_{\tau}-\tilde X^i_{\tau}|\le|x_i-y_i|E_xN^i_{0,\tau}=|x_i-y_i|$, it follows that
$|V_T(0,x)-V_T(0,y)|\le L|x-y|$ for all $T>0$. This and Corollary \ref{cor3.7} imply that we also have  $|V(x)-V(y)|\le L|x-y|$ for $x,y\in D$. \\
(ii) Since  $\psi(x)\le C(1+|x|)$, $x\in D$, for all  $T>0$ and $t\in[0,T]$ we have
$ V_T(t,x)\le V_T(0,x)
\le C+C\sup_{\tau\in\TT_{0,T}}E_xe^{-r\tau}|X_{\tau}|$.
Since $\delta_i\ge0$, $i=1,\dots,d$, for any $\tau\in\TT_{0,T}$ we also have
$|X_{\tau}|\le
\sum^d_{i=1}X^i_0e^{r\tau}N^i_{0,\tau}$. Since $E_xN^i_{0,\tau}=1$, $i=1,\dots,d$, this proves (ii).
\end{proof}

\subsection{Analytical characterization of the value function}
\label{sec3.2}

Our next aim is to show that the value function $V$ is the  unique
variational solution  of the semilinear problem
\begin{equation}
\label{eq3.43}
\begin{cases}
L_{BS}v=rv-\Phi(\cdot,v),\quad v\ge\psi\quad\mbox{ in }D,
\smallskip\\
\lim_{t\rightarrow\infty}e^{-rt}P_tv(x)=0,\quad x\in D.
\end{cases}
\end{equation}
Before formulating a precise definition of a solution of (\ref{eq3.43}), we first give some remarks on the connection between (\ref{eq3.43}) and the obstacle problem. Roughly speaking, since $\Phi$ is given by (\ref{eq2.9}), the first line
of (\ref{eq3.43}) means that
\begin{equation}
\label{eq3.44}
v\ge\psi\quad\mbox{and}\quad  L_{BS}v=rv+\begin{cases}
-\Psi^{-} & \mbox{on }\{v=\psi\},\\
0 & \mbox{on }\{v>\psi\}.
\end{cases}
\end{equation}
Note also that the measure $\nu$ on $D$ defined as
\[
\nu(dx)=\Psi^{-}(x)\fch_{\{v(x)=\psi(x)\}}\,dx
\]
has the property that
\begin{equation}
\label{eq3.45}
\int_D(v-\psi)(x)\nu(dx)=0.
\end{equation}
This means that  the pair $(v,\nu)$
is a solution of the so-called complementarity system associated with
the obstacle problem
\begin{equation}
\label{eq3.50}
\min\{v-\psi,-L_{BS}v+rv\}=0\quad\mbox{on }D.
\end{equation}
The ``minimality condition'' (\ref{eq3.45}) says that
$\nu$ (sometimes called the obstacle  reaction measure associated with the solution $v$ of (\ref{eq3.50}))
acts only when $v$ touches the obstacle $\psi$.
The fact that $(v,\nu)$ satisfies (\ref{eq3.44}), (\ref{eq3.45}) may be viewed as
an analytic counterpart to the first two lines of (\ref{eq6.1}). For more information about this kind of correspondence between reflected BSDE and solutions to complementarity systems associated with obstacle problems  see \cite{KR:AMO,KR:MF} (parabolic case) and \cite{K:SM,RS:EJP} (elliptic case).

Of course, to give a rigorous definition of (\ref{eq3.43}) (or (\ref{eq3.44})) we
have to specify in what sense the equation in the first line of (\ref{eq3.43}) is satisfied.  We are interested in solutions
of  (\ref{eq3.43}) in some Sobolev space. Since we require that
$v\ge\psi$ and $\psi$ satisfies (\ref{eq2.4}), it is natural to
work with Sobolev space with some weight $\varrho$ such that
$\int_{\BR^d}(1+|x|)^2\varrho^2(x)\,dx<\infty$. As in
\cite{KR:BPAN,KR:AMO}, our choice of the weight  is
$\varrho(x)=(1+|x|^2)^{-\gamma}$ with some  $\gamma>(2+d)/4$.
Then, by an elementary calculation,
$\int_{\BR^d}\varrho^2(x)\,dx<\infty$ and
$\int_{\BR^d}|x|^2\varrho^2(x)\,dx<\infty$. In particular, if
$\psi$ satisfies (A1) and $\Psi$ satisfies (\ref{eq3.42}), then
\begin{equation}
\label{eq3.35}
\int_{\BR^d}|\psi(x)|^2\varrho^2(x)\,dx<\infty,\qquad \int_{\BR^d}|\Psi^{-}(x)|^2\varrho^2(x)\,dx<\infty.
\end{equation}
Define
\[
L^2_{\varrho}(D)=L^2(D;\varrho^2\,dx),\qquad H^{1}_{\varrho}(D)=\{u\in
L^2_{\varrho}(D):\sum^d_{j=1} \sigma_{ij}x_iu_{x_j}\in
L^{2}_{\varrho},\,i=1,\dots,d\},
\]
and for $\phi,\varphi\in C^{\infty}_0(D)$ set
\begin{align*}
B^{BS}_{\varrho}(\phi,\varphi)&=\sum^d_{i=1}\int_{D}
(r-\delta_i)x_i\partial_{x_i}\phi(x)\varphi(x)\varrho^2(x)\,dx\\
&\quad-\frac12\sum^d_{i,j=1}\int_{D}
a_{ij}\partial_{x_i}\phi(x)\partial_{x_j}(x_ix_j\varphi(x)\varrho^2(x))\,dx.
\end{align*}
One can check that there is $c>0$ such that
\[
B^{BS}_{\varrho}(\phi,\varphi)
\le c\|\phi\|_{H^1_{\varrho}(D)}\|\varphi\|_{H^1_{\varrho}(D)}.
\]
Therefore the form $B^{BS}_{\varrho}$ can be extended to a bilinear form on
$H_{\varrho}(D)\times H_{\varrho}(D)$, which we still denote by $B^{BS}_{\varrho}$.
For an open set $U\subset\BR^d$, we define  the spaces $H^1(U)$,
$H^2(U)$ in the usual way.

\begin{definition}
We say that $v\in H^1_{\varrho}(D)$ is a variational  solution of
the semilinear problem
\begin{equation}
\label{eq3.32}
L_{BS}v=rv-\Phi(\cdot,v),\qquad v\ge\psi
\end{equation}
if $v(x)\ge\psi(x)$ for $x\in D$, $\Phi(\cdot,v)\in L^2_{\varrho}(D)$ and the equation in (\ref{eq3.32}) is satisfied in the weak sense, i.e. for every $\varphi\in H^1_{\varrho}(D)$,
\begin{equation}
\label{eq3.33}
B^{BS}_{\varrho}(v,\varphi)=(rv-\Phi(\cdot,v),\varphi)_{L^2_{\varrho}(D)}.
\end{equation}
\end{definition}
Recall that by Alexandrov's theorem (see, e.g., \cite[Theorem
7.10]{AA}), $\psi$  has  second order derivatives at $x$ for a.e.
$x\in\BR^d$. Consequently,  $\LL_{BS}\psi$  appearing in the definition of
$\Phi$ (see (\ref{eq2.11}))  is well defined for a.e. $x\in\BR^d$.
Of course, in general, $\psi$ does not have second
order derivatives in the distribution sense given by locally
integrable functions.

Below we show that under (A1), (A2) and (\ref{eq3.42}) a weak
solution of (\ref{eq3.32}) really exists, and in fact has second
order  derivatives in distribution sense given by locally
integrable functions. Note that by Remark \ref{rem3.1}, if
$\delta_i>0$, $i=1,\dots,d$, then (A1) and (3.5) imply (A2).

\begin{proposition}
\label{prop3.11} Assume that \mbox{\rm(A1), (A2)} and
\mbox{\rm(\ref{eq3.42})}  are satisfied. If $v$ is a variational
solution of  \mbox{\rm(\ref{eq3.32})}  then $v\in H^2_{loc}(D)$.
In particular,
\begin{equation}
\label{eq3.41}
L_{BS}v(x)=rv(x)-\Phi(x,v(x))\quad \mbox{for a.e. }x\in D.
\end{equation}
\end{proposition}
\begin{proof}
Fix a bounded open set $U$ such that $U\subset\bar U\subset D$.
Let $\xi\in C_0^{\infty}(U)$ and $\varphi=\xi/\varrho^2$.
Then $\varphi\in H^1_{\varrho}$, so from (\ref{eq3.33})
it follows that
\[
\BB^{BS}(v,\xi)=(rv-\Phi(\cdot,v),\xi)_{L^2(\BR^d;dx)},
\]
where $\BB^{BS}$ is defined as $B^{BS}_{\varrho}$ but with $\varrho=1$. Therefore $v$ is a weak solution, in the space $H^1(U)$,  of the problem $L_{BS}v=rv-\Phi(\cdot,v)$ in $U$.
To show that $v\in H^2_{loc}(D)$ we make a well known change of variables, which reduces  the study of (\ref{eq3.41}) to the study of an equation with  uniformly elliptic operator $\tilde L$ defined as
\[
\tilde L=\sum^d_{i=1}(r-\delta_i-a_{ii}/2)\partial_{x_i} +\frac12\sum^d_{i,j=1}a_{ij}\partial^2_{x_ix_j}.
\]
More precisely, write $e^x=(e^{x_1},\dots,e^{x_d})$ for  $x=(x_1,\dots,x_d)\in\BR^d$, and then define $\tilde v(x)=v(e^x)$, $\tilde \Phi(x)=\Phi(e^x,\tilde v(x))$ and  $\tilde U=\{x\in\BR^d:e^x\in U\}$.
An elementary computation shows that  $\tilde v\in H^1(\tilde U)$  and  $\tilde v$ is a weak solution of the problem
$\tilde L\tilde v=r\tilde v-\tilde\Phi$ in $\tilde U$. By \cite[Theorem 1, Section 6.3]{E}, $\tilde v\in H^2(\tilde U)$, from which it follows that $v\in H^2(U)$. Because of arbitrariness of $U$, $v\in H^2_{loc}(D)$. The equality (\ref{eq3.41}) now follows by a standard argument (see Remark (ii) following \cite[Section 6.3, Theorem 1]{E}).
\end{proof}

\begin{theorem}
\label{th3.13}
Assume that \mbox{\rm(A1), (A2)} and \mbox{\rm(\ref{eq3.42})} are satisfied. Then $V$ is a variational solution of \mbox{\rm(\ref{eq3.32})}.
\end{theorem}
\begin{proof}
Let $W_{\varrho}=\{u\in L^2(0,T;H^1_{\varrho}):u_t\in
L^2(0,T;H^{-1}_{\varrho})\}$. In  \cite{KR:AMO} it is proved that for every $T>0$,  $V_T\in W_{\varrho}$ and $V_T$  is a variational  solution of the Cauchy problem
\begin{equation}
\label{eq3.39}
\partial_tV_T+L_{BS}V_T=rV_T-\Phi(\cdot,V_T),\qquad V_T(T,\cdot)=\psi,
\end{equation}
i.e. $V_T\ge\psi$ and (\ref{eq3.39}) is satisfied in the weak sense. In particular,
for any test function $\eta\in C^{\infty}_0((0,T)\times D)$ we have
\[
\int^T_0\langle\partial_tV_T(t),\eta(t)\rangle\,dt +\int^T_0B^{BS}_{\varrho}(V_T(t),\eta(t))\,dt =\int^T_0(rV_T(t)-\Phi(\cdot,V_T(t)),\eta(t))_{L^2_{\varrho}}\,dt,
\]
where $V_T(t)=V_T(t,\cdot)$, $\eta(t)=\eta(t,\cdot)$ and $\langle\cdot,\cdot\rangle$ denotes the duality pairing between $L^2(0,T;H^{-1}_{\varrho})$ and  $L^2(0,T;H^1_{\varrho})$. From this one can deduce  that for every $\varphi\in C^{\infty}_0(D)$,
\begin{align*}
&\int^1_0\!\int_{D}\partial_tV_T(t,x)\varphi\varrho^2(x)\,dt\,dx
+\sum^d_{i=1}\int^1_0\!\int_{D}
(r-\delta_i)x_i\partial_{x_i}V_T(t,x)\varphi(x)\varrho^2(x)\,dt\,dx\\
&\qquad\qquad\qquad-\frac12\sum^d_{i,j=1}\int^1_0\!\int_{D}
a_{ij}\partial_{x_i}V_T(t,x)\partial_{x_j}(x_ix_j\varphi(x)\varrho^2(x))\,dt\,dx\\
&\qquad=\int^1_0\!\int_{D}(rV_T(t,x)-\Phi(x,V_T(t,x))
\varphi(x)\varrho^2(t,x)\,dt\,dx,
\end{align*}
that is
\begin{equation}
\label{eq3.34}
\int^1_0\langle\partial_tV_T(t),\varphi\rangle\,dt +\int^1_0B^{BS}_{\varrho}(V_T(t),\varphi)\,dt =\int^1_0(rV_T(t)-\Phi(\cdot,V_T(t)),\varphi)_{L^2_{\varrho}}\,dt.
\end{equation}
By Corollary \ref{cor3.7}(ii), for every $x\in D$, $V_T(0,x)\rightarrow V(x)$ and  $V_T(1,x)\rightarrow V(x)$. Furthermore, by Lemma 3.10(ii), $|(V_T(1,\cdot)-V(0,\cdot))\varphi|\varrho^2$ is bounded by the function $x\mapsto
2C(1+|x|)|\varphi(x)|\varrho^2(x)$,  which is integrable on $D$ since $\varphi\in C^{\infty}_0(D)$. Therefore applying the dominated convergence theorem we get
\[
\lim_{T\rightarrow\infty}\int^T_0(\partial_tV_T(t),\varphi)_{L^2_{\varrho}}\,dt
=\lim_{T\rightarrow\infty}\int_{D}(V_T(1,x)-V_T(0,x))\varphi(x)\varrho^2(x)\,dx=0.
\]
Suppose that $\mbox{supp}[\varphi]\subset U$ for some relatively compact open set $U\subset D$. By Lemma \ref{lem3.8}, $|\partial_{x_i}V_T|\le L$ a.e. for all $i=1,\dots,d$ and $T>0$, and $V_T$ are bounded on $(0,1)\times U$ uniformly in $T>0$. By this and Corollary \ref{cor3.7}(ii), $V_T\rightarrow V$ weakly in $L^2(0,1;H^1(U))$. Therefore, for $i=1,\dots,d$, we have
\begin{align*}
&\lim_{T\rightarrow\infty}
\int^1_0\int_{D}(r-\delta_i)x_i\partial_{x_i}V_T(t,x)\varphi(x)\varrho^2(x)\,dt\,dx\\
&\qquad
=\int_{D}(r-\delta_i)x_i\partial_{x_i}V(x)\varphi(x)\varrho^2(x)\,dx\\
\end{align*}
and
\begin{align*}
&\lim_{t\rightarrow\infty}\sum^d_{j=1}\int^1_0\int_{D}a_{ij}x_i\partial_{x_i}V_T(t,x)
\partial_{x_i}(x_ix_j\varphi(x)\varrho^2(x))\,dt\,dx\\
&\qquad=\sum^d_{j=1}\int_{D}a_{ij}x_i\partial_{x_i}V(x)
\partial_{x_i}(x_ix_j\varphi(x)\varrho^2(x))\, dx.
\end{align*}
Hence
\begin{equation}
\label{eq3.36}
\lim_{T\rightarrow\infty}\int^1_0B^{BS}_{\varrho}(V_T(t),\varphi)\,dt =B^{BS}_{\varrho}(V,\varphi).
\end{equation}
Since $V_T\le V_{T'}$ if $T\le T'$, in fact $V_T\nearrow V$  as $T\rightarrow\infty$. Therefore $\Phi(\cdot,V_T)\rightarrow\Phi(\cdot,V)$ pointwise. Furthermore,
by Lemma 3.10(ii), $V_T(t,x)\le C(1+|x|)$, and by the definition of $\Psi$ we have $\Phi(x,V_T(t,x))\le\Psi^{-}(x)$, $(t,x)\in(0,1)\times D$. Hence $|(rV_T-\Phi(\cdot,V_T))\varphi|\varrho^2$ is bounded by the function $x\mapsto (rC(1+|x|)+\Psi^{-}(x))|\varphi(x)|\varrho^2(x)$, which is integrable on $(0,1)\times D$ by (\ref{eq3.35})  and the fact that $\varphi\in C^{\infty}_0(D)$. Therefore applying the  dominated convergence theorem we get
\begin{align*}
&\lim_{T\rightarrow\infty}
\int^1_0\int_{D}(rV_T(t,x)-\Phi(x,V_T(t,x))\varphi(x)\varrho^2(x)\,dt\,dx\\
&\qquad=\int_{D}(rV(x)-\Phi(x,V(x)))\varphi(x)\varrho^2(x)\,dx,
\end{align*}
i.e.
\begin{equation}
\label{eq3.37}
\lim_{T\rightarrow\infty}\int^1_0(rV_T(t)-\Phi(\cdot,V_T(t)),\varphi)_{L^2_{\varrho}}\,dt
=(rV-\Phi(\cdot,V),\varphi)_{L^2_{\varrho}}.
\end{equation}
From (\ref{eq3.34})--(\ref{eq3.37}) it follows that  $V$ satisfies (\ref{eq3.33}) for $\varphi\in C^{\infty}_0(D)$, and hence for  $\varphi\in H^1_{\varrho}$ by
an approximation argument. Clearly $V\ge\psi$, so $V$ is a solution of (\ref{eq3.32}).
\end{proof}

Before stating the uniqueness result, we note that  under the assumptions on $\psi$ and $\delta_1,\dots, \delta_d$ stated in Remark \ref{rem3.1}(ii), $e^{-rt}P_tV(x)\rightarrow0$ as $\rightarrow\infty$.
Therefore it is natural to prove uniqueness in the class of functions having the same property.

\begin{proposition}
Under the assumptions of Theorem \ref{th3.13}  there is at most one variational solution $v$ of \mbox{\rm(\ref{eq3.43})}.
\end{proposition}
\begin{proof}
Let $v^1,v^2$ be two solutions of (\ref{eq3.43}), and let $v=v^1-v^2$. Define $\tilde L$ as in the proof of Proposition \ref{prop3.11} and set $\tilde v(x)=v(e^x)$. Then
$v(X)=\tilde v(Z)$, where $Z=(Z^1,\dots,Z^d)$,
$Z^i_t=\ln x_i+(r-\delta_i-a_{ii}/2)t+\sum^d_{j=1}\sigma_{ij}B^j_{0,t}$, $t\ge0$.
Choose an increasing  sequence  $\{U_n\}$ of bounded open sets such that $\bar U_n\subset U_{n+1}$ and $\bigcup_{n\ge1}U_n=D$ and set  $\tau_n=\inf\{t>0:X_t\notin  U_n\}=\inf\{t>0:Z_t\notin \tilde U_n\}$, where $\tilde U_n=\{x\in\BR^d:e^x\in U\}$. Since $\tilde v\in H^2(\tilde U_n)$, by  the extension of It\^o's formula proved by Krylov (see \cite[Chapter II, \S10, Theorem 1]{Kr}) we have
\[
\tilde v(Z_{t\wedge\tau_n})=\tilde v(Z_0)+\sum^d_{i,j=1}\int^{t\wedge\tau_n}_0\partial_{x_i}\tilde v(Z_s)\,\sigma_{ij}\,dB^j_{0,s}
+\int^{t\wedge\tau_n}_0\tilde L\tilde v(Z_s)\,ds,\quad t\ge0.
\]
Define $Y_t=v(X_t)$, $t\ge0$. Since $v(X)=\tilde v(Z)$, it follows that
\begin{equation}
\label{eq3.38}
Y_{t\wedge\tau_n}=Y_0+\sum^{d}_{i,j=1}\int^{t\wedge\tau_n}_0L_{BS}v(X_s)\,ds +R_{t\wedge\tau_n},\quad t\ge0,
\end{equation}
where $R_t=\sum^d_{i,j=1}\int^t_0\sigma_{ij}X^i_s\partial_{x_i}v(X_s)\,dB^j_{0,s}$.
Since $P_x(X_t\in D,t\ge0)=1$, $\tau_n\rightarrow\infty$ $P_x$-a.s. as $n\rightarrow\infty$. Therefore  letting $n\rightarrow\infty$ in (\ref{eq3.38}) shows that it  holds true with $t\wedge\tau_n$ replaced by $t$. Let $\bar Y_t=e^{-rt}Y_t$. Integrating by parts we obtain
\begin{align*}
\bar Y_t&=\bar Y_0+\int^t_0(-re^{-rs}Y_s\,ds+\int^t_0e^{-rs}\,dY_s\\
&=\bar Y_0+\int^t_0e^{-rs}(-rv+L_{BS}v)(X_s)\,ds+\int^t_0e^{-rs}\,dR_s\\
&=\bar Y_0-\int^t_0e^{-rs}(\Phi(X_s,v^1(X_s))-\Phi(X_s,v^2(X_s))\,ds+\int^t_0e^{-rs}\,dR_s.
\end{align*}
Repeating now the argument from the proof of  Proposition \ref{prop2.1} we show that
$E_x\bar Y^+_0\le E_x\bar Y^+_t$, $t\ge0$. In much the same way we show that $E_x\bar Y^-_0\le E_x\bar Y^-_t$, $t\ge0$. Hence $E_x|\bar Y_0|\le E_x|\bar Y_t|=e^{-rt}E_x|v(X_t)|=e^{-rt}P_t|v|(x)$, which converges to zero as $t\rightarrow\infty$. Thus $|v(x)|=E_x\bar Y_0=0$.
\end{proof}

In the case of American  call  and American put on single asset
explicit formulas for the solution of (\ref{eq3.43}) are known
(see, e.g., \cite{J,KS,McK,S}).  Assume  that $d=1$ and
$\psi=(K-x)^+$, $x\in\BR$.  Then from (\ref{eq2.9}) and
(\ref{eq2.11}) it follows that
\[
\Phi(x,v(x))=\begin{cases}
(rK-\delta x)^+ &\mbox{if }v(x)\le\psi(x), \\
0 &\mbox{if }v(x)>\psi(x).
\end{cases}
\]
Let $v$ be a variational solution of (\ref{eq3.43}). Then
$v\ge\psi$ and $v$ satisfies the equation
\begin{equation}
\label{eq3.46} L_{BS}v=rv+\begin{cases}
-(rK-\delta x)^+ &\mbox{on }\{ v=\psi\}, \\
0 &\mbox{on }\{v>\psi\},
\end{cases}
\end{equation}
in the weak sense (see the definition preceding Proposition
\ref{prop3.11}). Furthermore, by Proposition \ref{prop3.11}, $v\in
H^2_{loc}(D)$ and  (\ref{eq3.46}) is satisfied for a.e.
$x\in(0,\infty)$. In fact much more can be said. McKean \cite{McK}
(see also \cite{J} and \cite[Chapter 2, Theorem 7.2]{KS}) showed that $v$ has
the form
\begin{equation}
\label{eq3.51}
v(x)=\begin{cases}
K-x, & 0\le x\le b,\\
(K-b)(x/b)^{\gamma}, &\quad x> b,
\end{cases}
\end{equation}
where $\gamma=-(1/\sigma)(\nu+\sqrt{\nu^2+2r})$,
$\nu=-(1/2)\sigma+(r-\delta)/\sigma$ and $b=\gamma K/(\gamma-1)$.
In particular, we see that $v$ is $C^1((0,\infty)\cap C^2((0,\infty)\setminus\{b\})$
and $\{v>\psi\}=(b,\infty)$, $\{v=\psi\}=[0,b]$. Furthermore, from (\ref{eq3.51}) it follows that (\ref{eq3.46}) is satisfied for every $x\in(0,b)\cup(b,\infty)$.
For the corresponding formulas for $v$ in the case of American call we refer the reader to
\cite[Theorem 6.7]{KS}. Let $\CC$ denote the continuation region for the stopping problem (\ref{eq1.2}) with $s=0$, that is $\CC=\{(t,x)\in[0,\infty)^2:V_t(0,x)>\psi(x)=(K-x)^+\}$, and let $\CC_t$ be the $t$ section of $\CC$, i.e. $\CC_t=\{x\ge0:(t,x)\in \CC\}$.
In \cite{J} (see also \cite[Section 2.7]{KS}) it is proved that
$\CC_t=(b(t),\infty)$, $t>0$ for some continuous function $b(t)$ called the free boundary for the parabolic obstacle problem (\ref{eq5.3}). Furthermore, the constant $b$ of (\ref{eq3.51}) is the limit, as $t\rightarrow\infty$, of $b(t)$.

\section{Examples}

Below we give examples of payoff functions satisfying (A1), (A2)
and (\ref{eq3.42}). In all the examples $\Psi^{-}$ is computed in
the subset $D\cap\{\psi>0\}$ (see Remark \ref{rem3.8}(ii)).

\begin{example}
\label{ex4.1} Let $d=1$.
\[
\psi(x)=(x-K)^+,\qquad \Psi^{-}(x)=(\delta x-rK)^+\quad
\mbox{(call)}
\]
\[
\psi(x)=(K-x)^+,\qquad \Psi^{-}(x)=(rK-\delta x)^+\quad
\mbox{(put)}
\]
The assumptions (A1) and  (A2) are satisfied if  $r>0$ in case of put option, and if  $r>0,\delta>0$ in case of call option. By (\ref{eq3.40}), for put option we have
\[
V(x)-V_T(0,x)\le e\Big(Ke^{-rT}+rK\int^{\infty}_Te^{-rt}\,dt\Big)=2eKe^{-rT},\quad x>K.
\]
For call option, $V(x)-V_T(0,x)\le 2eKe^{-\delta T}$, $T>0$, $x\in(0,K)$.
\end{example}

\begin{example}
\label{ex4.2} In the examples below $d\ge 2$. In all the cases
where $\psi$ is bounded, (A1) and (A2) are satisfied if $r>0$.  In
the other cases they are satisfied if $r>0$ and $\delta_i>0$,
$i=1,\dots,d$.
\begin{enumerate}[(i)]
\item Index options and spread options.
\[
\psi(x)=\big(\sum_{i=1}^{d} w_{i}x_i-K\big)^{+},\quad
\Psi^-(x)=\big(\sum_{i=1}^{d} w_{i}\delta_{i}x_i-r K\big)^{+}\quad
\mbox{(call)}
\]
\[
\psi(x)=\big(K-\sum_{i=1}^{d} w_{i}x_i\big)^{+},\quad
\Psi^-(x)=\big(r K-\sum_{i=1}^{d} w_{i}\delta_{i}x_i\big)^{+}\quad
\mbox{(put)}
\]
\item Call on max option.
\[
\psi(x)=(\max\{x_1,\dots,x_d\}-K)^{+},\qquad
\Psi^{-}(x)=\big(\sum_{i=1}^d\delta_{i}\mathbf{1}_{B_{i}}(x)x_i-r
K\big)^{+},
\]
where $B_{i}=\{x\in\BR^d: x_{i}>x_{j},\, j\neq i\}$.

\item Put on min option.
\[
\psi(x)=(K-\min\{x_1,\dots,x_d\})^{+},\qquad
\Psi^{-}(x)=\big(r K
-\sum_{i=1}^d\delta_{i}\mathbf{1}_{C_{i}}(x)x_i\big)^{+},
\]
where $C_{i}=\{x\in\BR^d: x_{i}<x_{j},\, j\neq i\}$.

\item Multiple strike options.
\[
\psi(x)=(\max\{x_1-K_{1},\dots, x_d-K_d\})^{+},
\]
\[
\Psi^{-}(x)=\big(\sum_{i=1}^{d} \mathbf{1}_{B_{i}}(x-K)(\delta_{i}x_i-r
K_{i})\big)^{+}\quad \mbox{ with }K=(K_{1},\dots,K_d).
\]
\end{enumerate}
\end{example}

\begin{example}
No explicit solution of (\ref{eq3.32}) seem possible in the
multidimensional cases considered in Example \ref{ex4.2}. Note
however, that (\ref{eq3.20}) gives an integral formula for $V$.
For instance, in case $d=2$ and $\psi(x)=(\max\{x_1,x_2\}-K)^+$
(see Example \ref{ex4.2}(ii)), we have
\begin{align}
\label{eq4.1}
V(x)&=E_x\int^{\infty}_0e^{-rt}\big\{(\delta_1X^1_t-rK)^+\fch_{\{X^1_t>X^2_t\}}
\fch_{\{V(X_t)=X^1_t-K>0\}} \nonumber \\
&\qquad\qquad\qquad
+(\delta_2X^2_t-rK)^+\fch_{\{X^2_t>X^1_t\}}\fch_{\{V(X_t)=X^2_t-K>0\}}\big\}
\,dt.
\end{align}
As in the case of options with finite exercise time (see
\cite[Proposition 2.7]{BD} and the remarks following it), formula
(\ref{eq4.1}) (and similar formulas for other options considered
in Example \ref{ex4.2}) has the potential to be used in a
numerical valuation procedure. However, as remarked in \cite{BD},
its implementation may be a challenge.
\end{example}

\subsection*{Acknowledgements}
{This work was supported by the Polish National Science Centre
under Grant  \\ 2016/23/B/ST1/01543).}

\end{document}